\documentclass[a4paper,11pt]{amsart}
\newtheorem{definition}{Definition}
\newtheorem{remark}{Remark}
\newtheorem{theorem}{Theorem}
\newtheorem{lemma}{Lemma}
\newtheorem{question}{Question}

\newtheorem{corollary}{Corollary}

\usepackage{fullpage}
\usepackage{graphicx}
\usepackage{amssymb}
\usepackage{hyperref}
\usepackage{mathtools}

\usepackage{color}
\usepackage{soul}

\newcommand{\cU}{\mathcal{U}}
\newcommand{\id}{\mathrm{id}}

\title{Time-varying Spaces and Mobile Sensor Networks}
\author{Tia Karkos}
\email{tia.karkos@colostate.edu}
\author{Henry Adams}
\email{henry.adams@ufl.edu}
\date{\today}

\begin{document}

\begin{abstract}
Consider a mobile sensor network, in which each sensor covers a ball.
Sensors do not know their locations, but can detect if the covered balls overlap.
An intruder cannot pass undetected between overlapping sensors.
An evasion path exists if it is possible for an intruder to move in the domain without ever entering a covered region.
We examine two time-varying topological spaces arising from such mobile sensor networks.
These examples were constructed by Adams and Carlsson~\cite{EvasionPaths} to show that the time-varying homotopy type of the covered region does not determine whether an evasion path exists or not.
One of these spaces has an evasion path and the other does not, which means that the uncovered regions are not time-varying homotopy equivalent.
We show that the covered regions of these spaces are not time-varying homeomorphic, even though they are time-varying homotopy equivalent.
We then elaborate on this time-varying homotopy equivalence between the covered regions.
\end{abstract}


\maketitle

\section{Introduction}

Let $D\subseteq \mathbb{R}^d$ be homeomorphic to a closed $n$-dimensional ball.
Suppose we have a collection of (smaller) ball-shaped sensors in $D$ that cover the boundary $\partial D$; see Figure~\ref{fig:sensor-balls}.
The sensors do not know their exact locations (as measured by GPS for example), but two sensors can detect if they overlap.
Using only this local connectivity data, de Silva and Ghrist~\cite{Coordinate-free,de2007coverage} show how to compute homology in order to decide if the sensors cover all of the domain $D$ or not.
This is an example of a \emph{minimal sensing} algorithm, since each sensor measures only a small amount of local connectivity data.
Nevertheless, we can ``integrate'' these local measurements together using homology to answer a global coverage question.

\begin{figure}[htb]
\centering
\includegraphics[width=2.5in]{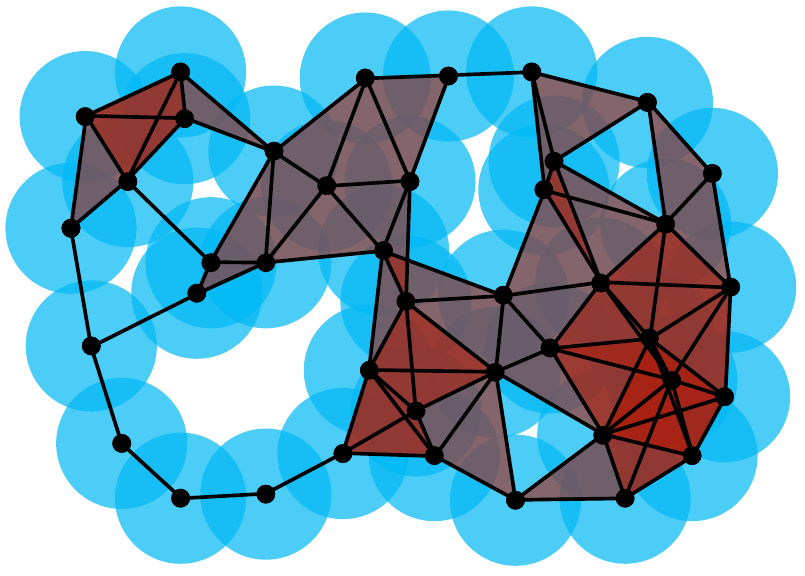}
\caption{Sensor balls in a planar domain.}
\label{fig:sensor-balls}
\end{figure}

In~\cite{Coordinate-free}, de Silva and Ghrist furthermore consider the setting of mobile sensors.
Following their lead, we assume that there is a fixed collection of sensors covering the boundary $\partial D$ of the domain.
All other sensors are allowed to wander continuously in $D$.
Again, the sensors do not know their exact time-varying locations, but two sensors can detect when they overlap.
Using only this local connectivity data, can we determine if the sensors necessarily detect any continuously moving mobile intruder?
That is, using only this local connectivity data, can we determine if there exists an \emph{evasion path} avoiding the sensors or not?
Interestingly, this question about mobile sensors is more subtle than the setting of static sensors.
Indeed, Adams and Carlsson~\cite{EvasionPaths} show that the time-varying connectivity data alone does not determine whether an evasion path exists or not.
They construct two sensor networks with the same time-varying connectivity data, and whose covered regions are time-varying homotopy equivalent, in which one sensor network has an evasion path while the other does not.
In other words, the time-varying homotopy type of a sensor network's covered region does not determine if an evasion path exists or not; how that covered region is embedded in spacetime also matters.

The purpose of this paper is three-fold.
First, we prove that the two sensor network examples constructed by Adams and Carlsson have covered regions that are \emph{not} time-varying homeomorphic, even though they are time-varying homotopy equivalent.
Second, we give a detailed explanation of the time-varying homotopy equivalence claimed by Adams and Carlsson.
Third, we ask the following question.
In any example of two sensor networks with covered regions that are time-varying homotopy equivalent and uncovered regions that are not time-varying homotopy equivalent (for example if one has an evasion path and the other does not), is it necessarily the case that the covered regions are not time-varying homeomorphic?

We begin in Section~\ref{sec:related} by surveying related work, before covering preliminaries and mathematical background in Section~\ref{sec:preliminaries}.
We remind the reader of notions of equivalence (homeomorphism and homotopy equivalence) between spaces in Section~\ref{sec:time-varying}, before then overviewing the generalizations of these notions to the setting of time-varying spaces.
In Section~\ref{sec:application} we apply these notions of time-varying equivalences between spaces to examples arising from mobile sensor networks.
We conclude with some open questions in Section~\ref{sec:conclusion}.

\section{Related work}
\label{sec:related}

The sensor network problem that we study was first introduced by de Silva and Ghrist.
In the setting of static sensors, de Silva and Ghrist show that connectivity data alone can determine if the entire domain is covered or not~\cite{Coordinate-free,de2007coverage}.
Indeed, for sensors in a ball-shaped domain in $d$-dimensional Euclidean space, one can use the connectivity data of the sensors to compute the $(d-1)$-dimensional homology of the covered region.
Since the boundary of this region is covered by sensors, there is a hole in the covered region if and only if the $(d-1)$-dimensional homology of the covered region is nonzero.

Furthermore, in~\cite{Coordinate-free} de Silva and Ghrist introduce a variant of this problem for mobile sensor networks.
The sensors are now moving continuously inside the domain, although there are fixed sensors covering the boundary.
One is interested if there is an evasion path or not, i.e., if it is possible for a continuously moving intruder to evade the sensors.
By using stacked Vietoris--Rips or stacked \v{C}ech complexes, de Silva and Ghrist provide a one-sided criterion (relying on relative homology) that, in certain situations, can guarantee that no evasion path can exist.
However, this is not an if-and-only if criterion.

Surprisingly, Adams and Carlsson~\cite {EvasionPaths,MyThesis} demonstrate that the time-varying \v{C}ech complex and the fiberwise homotopy type of the covered region of a mobile sensor network are insufficient to determine whether an evasion path exists (Figure~\ref{fig:spaces}).
This is a sharp departure from the case of static sensor networks.
Indeed, mobile coverage also depends on how the covered region of the mobile sensor network is embedded into spacetime.
Nevertheless, Adams and Carlsson use zigzag persistence to provide a one-sided criterion for mobile coverage that is no stronger than the prior condition by de Silva and Ghrist, but which can be computed in a streaming fashion.

Adams and Carlsson also show that for planar sensor networks, one can use weak rotation data (embedding data) and the time-varying alpha complex to determine if-and-only-if an evasion path exists.
Computing the alpha complex requires more geometric information than one would ideally be asked to measure in the context of minimal sensing.
This motivates the following question: Can the same weak rotation data be incorprated with the time-varying \v{C}ech complex (which is easier to measure than the alpha complex) to determine if-and-only-if an evasion path exists?
See the Open Question on Pages~29 and~109 of~\cite{MyThesis} and~\cite{EvasionPaths}, respectively.
This question remains open a decade later.

For other work on coverage problems in sensor networks, including some recent work on mobile sensors, we refer the reader to~\cite{adams2021efficient,carlsson2020space,cavanna17when,chintakunta2014distributed,gamble2015coordinate,Kerber2013,distributed}.

What we call \emph{time-varying spaces} are in other contexts often called \emph{fiberwise spaces}.
(In a fiberwise space, the base space no longer needs to be an interval of time, but could be an arbitrary topological space, and the map to the base space is often assumed to be a \emph{fibration}.
We will not need the concept of fibration maps in this paper.)
We refer the reader to~\cite{crabb2012fibrewise} for more information on fiberwise spaces.

\section{Preliminaries}
\label{sec:preliminaries}

This section contains preliminaries on the fundamental group, the \v{C}ech complex, the nerve lemma, and the sensor network problem.
We refer the reader to the books~\cite{armstrong2013basic,EdelsbrunnerHarer,Hatcher} for more details on topology.

\subsection{Fundamental Groups}

Given a topological space $X$ and a point $x_0$, the fundamental group $\pi_1 (X, x_0)$ is the set of all homotopy equivalent loops in $X$ that start and end at $x_0$.
If $\alpha$ is a loop in $X$, then we let $[\alpha]$ denote its equivalence class.
The fundamental group is a \emph{functor} from topological spaces to groups, which implies that if $f\colon X \to Y$ is a continuous map between two topological spaces, then there is an induced map $\pi_1(f) \colon \pi_1(X,x_0) \to \pi_1(Y,f(x_0))$ that is a group homomorphism.
We often simplify notation by omitting basepoints, writing $\pi_1(X)$ instead of $\pi_1(X,x_0)$.

\begin{figure}[htb]
\centering
\includegraphics[width=4in]{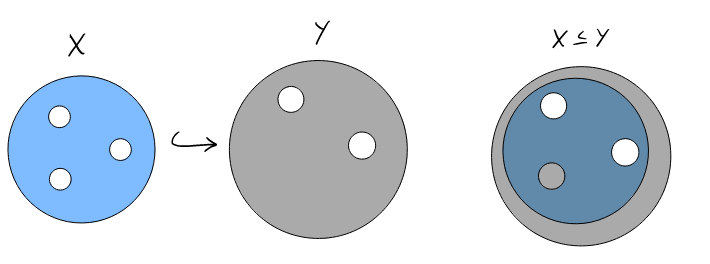}
\caption{Embedding of $X$ in $Y$.}
\label{fig:inclusion}
\end{figure}

\begin{figure}[htb]
\centering
\includegraphics[width=4in]{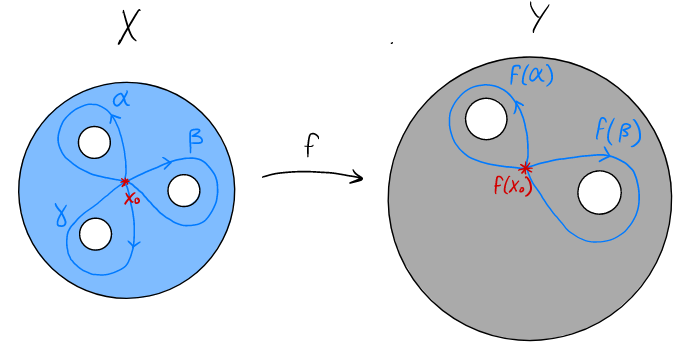}
\caption{Induced map on the fundamental group $\pi_1$.}
\label{fig:inclusionloops}
\end{figure}

For example, Figure~\ref{fig:inclusion} shows spaces $X$ and $Y$ and an embedding of $X$ in $Y$. In Figure~\ref{fig:inclusionloops}, we see that the induced map $f$ on the fundamental group $\pi_1$ sends the loops $\alpha$ and $\beta$ to the corresponding loops $f(\alpha)$ and $f(\beta)$ in $Y$.
But because the hole that the loop $\gamma$ surrounds gets filled in, $f(\gamma)$ can be contracted in $Y$.
So, $\pi_1(X)=\langle[\alpha],[\beta],[\gamma]\rangle$ is the free group on three generators (corresponding to the three holes in $X$), $\pi_1(Y)=\langle [f(\alpha)],[f(\beta)]\rangle$ is the free group on two generators, and the inclusion $f\colon X \hookrightarrow Y$ induces a map on the fundamental group that is generated by sending the homotopy class of $\alpha$ to the homotopy class of $f(\alpha)$, the homotopy class of $\beta$ to the homotopy class of $f(\beta)$, and the homotopy class of $\gamma$ to the identity element in $\pi_1(Y)$.

\subsection{The Nerve Lemma and \v{C}ech Complex}
Let $X$ be a topological space and let $\cU$ be a cover of $X$ via open sets.
The fact that $\cU$ is a \emph{cover} of $X$ means that for each $x\in X$, there exists some set $U\in\cU$ with $x\in U$.

\begin{definition}
\label{def: nerve simplicial complex}
The nerve simplicial complex $N(\cU)$ has a vertex for each set in $\cU$, and a $k$-simplex when the corresponding $k+1$ sets in $\cU$ have a non-empty $(k+1)$-fold intersection.
\end{definition}

\begin{figure}[htb]
\centering
\includegraphics[width=3in]{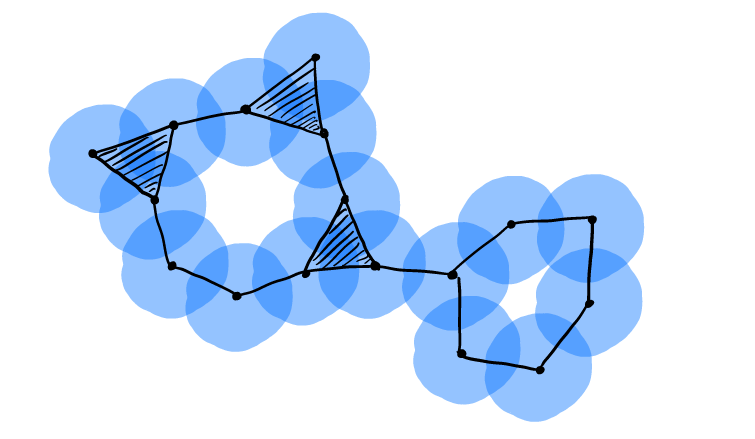}
\caption{Nerve simplicial complex for sensor balls.}
\label{fig:sensors}
\end{figure}

The nerve lemma is an important result which, under certain criteria, can guarantee that the nerve simplicial complex $N(\cU)$ is homotopy equivalent\footnote{
I.e., ``has the same shape as''; see Section~\ref{sssec:homotopy-equiv}.
}
to the space $X$.
See Figure~\ref{fig:sensors}.
We refer the reader to~\cite{Borsuk1948,EdelsbrunnerHarer,Hatcher} for more background on the nerve lemma.

\begin{theorem}[Nerve lemma]
\label{thm: nerve lemma}
Let $X$ be a paracompact topological space, and let $\cU$ be an open cover of $X$.
If the cover $\cU$ is \emph{good}, meaning that each set in $\cU$ is contractible and each intersection of sets in $\cU$ is either empty or contractible, then we have a homotopy equivalence $X \simeq N(\cU)$.
\end{theorem}

The assumption that $X$ is \emph{paracompact} is a point-set topology condition: the relevant property is that an open cover of a paracompact topological space admits a \emph{partition of unity}, which is used in the proof of the nerve lemma.
Every metric space $X$ is paracompact.

One instance in which we can apply the nerve lemma is the following situation.
Fix a radius $r>0$.
Suppose $Z=\{z_1,\ldots,z_n\}$ is a collection of points in Euclidean space $\mathbb{R}^d$.
Let $B(z_i;r)=\{y\in \mathbb{R}^d~|~\|z_i-y\| < r\}$ be the open ball of radius $r$ about $z_i$.
Let $X = \cup_{i=1}^n B(z_i;r)$ be the union of balls.
The \emph{\v{C}ech complex of $Z$ at scale $r$} is the nerve simplicial complex of the cover of $X$ by the collection of balls $\cU=\{B(z_i;r)\}_{i=1}^n$.
Each ball is contractible (because it is convex), and each intersection of balls is either empty or contractible (because an intersection of convex sets is convex); see Figure~\ref{fig:sensors}.
Therefore, the nerve lemma applies, and says that the \v{C}ech complex of $Z$ at scale $r$ is homotopy equivalent to $X$, the union of the balls.

\subsection{The Sensor Network Problem}
Suppose we have a network of sensors where each sensor covers a ball.
We can ask whether the sensor network covers the domain or not.
Also, what are the minimal sensing capabilities that we need in order to determine if the sensor network covers the domain?

Let us set notation in order to make things more precise.
Let $D\subseteq \mathbb{R}^d$ be a bounded domain that is homeomorphic to a closed ball.
Let $\partial D$ be its boundary.
For example, we could let the domain be the closed unit ball $D=\{x\in\mathbb{R}^d~|~\|x\|\le1\}$, in which case the boundary of this domain would be the unit sphere $\partial D=\{x\in\mathbb{R}^d~|~\|x\|=1\}$.
For the static version of the sensor network problem, let $Z \subseteq D$ be a fixed finite set of sensor locations.
About each sensor $z\in Z$ we have a ball $B(z,r)$ of radius $r$ that is observed; see Figure~\ref{fig:sensor-balls}.
So inside of our domain, the covered region is $X \coloneqq \cup_{z\in Z} B(z,r) \cap D$, and the uncovered region is $X^c \coloneqq D \setminus X$.
We assume that the sensors cover the boundary of the domain, namely $\partial D \subseteq X$.
Suppose the sensors can mesure the \v{C}ech complex.
Can we use this connectivity information measured by the sensor balls to decide if all of $D$ is covered by sensors or not?
The answer is yes, using the $(d-1)$-dimensional homology of the \v{C}ech complex~\cite{Coordinate-free,de2007coverage}.

We next describe the \emph{mobile} or \emph{time-varying} sensor network problem.
Suppose now that the sensors are moving.
Let $I=[0,1]$ be the time interval.
Let $Z_t \subseteq D$ be the finite set of locations of the sensors at any fixed time $t\in I$.
At each time $t$, we have that the region covered by the sensors is $X_t \coloneqq \cup_{z_t\in Z_t} B(z_t,r) \cap D$ and the uncovered region is $X^c_t \coloneqq D \setminus X_t$.
Now we have that the covered region in spacetime is $X \coloneqq \cup_{t\in I} X_t \times \{t\} \subset D \times I$.
The uncovered region in spacetime is $X^c \coloneqq (D \times I) \setminus X$, which can equivalently be written as $X^c = \cup_{t\in I} X^c_t \times \{t\} \subset D \times I$.
In a mobile sensor network, we can ask whether an evasion path exists; that is, whether an intruder could remain in the domain undetected during a time interval $I$ of interest.
Formally, an \emph{evasion path} is a continuous map $p \colon I \to D$ such that $p(t)$ is in the uncovered region $X_t^c$ for all $t\in I$.
(Later, in Definition~\ref{def: tv map}, we will see this is equivalent to a \emph{time-varying map} $I \to X^c\subseteq D\times I$.)
We say that an evasion path does not exist if any continuous map $p\colon I \to D$ necessarily satisfies that $p(t)$ is in the covered region $X_t$ for some $t\in I$.

\section{Notions of equivalence between time-varying spaces}
\label{sec:time-varying}

We review notions of equivalence between spaces in Section~\ref{sec:equivalence}, before introducing notions of equivalence between time-varying spaces in Section~\ref{sec:tv-equivalence} and between pairs of spaces in Section~\ref{sec:pair-equivalence}.
In each section, we discuss both homeomorphisms and homotopy equivalences.

\subsection{General Notions of Equivalence}
\label{sec:equivalence}

We will begin by defining two notions of equivalence between topological spaces, homeomorphism and homotopy equivalence.
Later, we will describe time-varying spaces and generalize our definitions of homotopy equivalence and homeomorphism to suit this type of space.

\subsubsection{Homeomorphism}
We begin with homeomorphisms.

\begin{definition}\label{def:homeomorphism}
Two topological spaces $X$ and $Y$ are \emph{homeomorphic (denoted $X \cong Y$)} if there exists a continuous  bijection $f\colon X \to Y$ with continuous inverse $f^{-1}\colon Y\to X$.
If this is the case, then we say that the map $f$ is a \emph{homeomorphism}.
\end{definition}

For what follows, it will be useful to introduce slightly different notation for homeomorphisms.
Suppose that we are given continuous maps $f\colon X\to Y$ and $g\colon Y\to X$ such that the compositions $g\circ f=\id_X$ and $f\circ g=\id_Y$ are the respective identity maps on $X$ and on $Y$.
It then follows that $f$ and $g$ are bijective maps, and that $g=f^{-1}$.
So, if such continuous maps $f$ and $g$ exist, then $X$ and $Y$ are homeomorphic.

\begin{figure}[htb]
\centering
\includegraphics[width=4.1in]{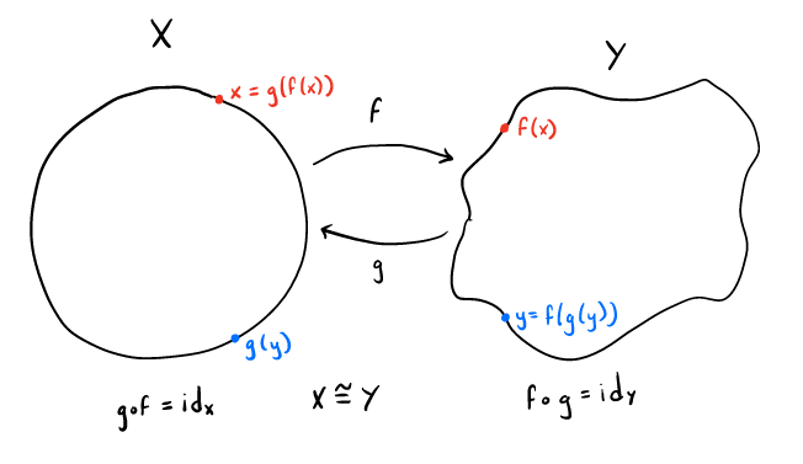}
\caption{Homeomorphism between topological spaces.}
\label{fig:homeomorphism}
\end{figure}

For example, see Figure~\ref{fig:homeomorphism}, which depicts a circle and a deformed circle.
If we map each point in $X$ to the corresponding point in $Y$ and vice versa via the continuous maps $f\colon X \to Y$ and $g\colon Y\to X$, we have $f \circ g = \id_Y$ and $g \circ f = \id_X$.
Thus, both $f$ and $g$ are homeomorphisms, with each other as inverses.

The reason why we introduce the notation of the map $g$ (instead of $f^{-1}$) is that we will relax the requirement that $f$ is a bijection when we next introduce the concept of homotopy equivalence.

\subsubsection{Homotopy Equivalence}
\label{sssec:homotopy-equiv}

Homotopy equivalences are a broad generalization of homeomorphisms to non-bijective maps.
Indeed, homotopy equivalence is a weaker notion of equivalence between spaces: any two homeomorphic spaces are homotopy equivalent, but the converse is not true in general.

\begin{figure}[htb]
\centering
\includegraphics[width=3.5in]{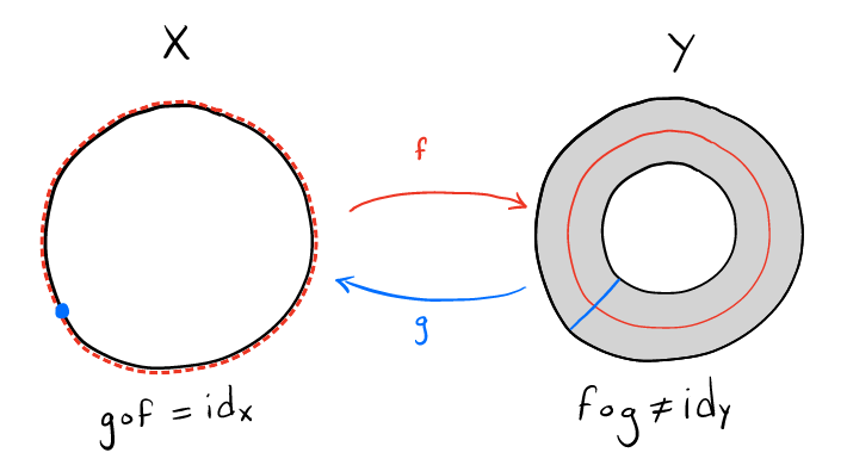}
\caption{Homotopy equivalence between topological spaces.}
\label{fig:homotopy}
\end{figure}

We begin with an example homotopy equivalence in Figure~\ref{fig:homotopy}, in which the space $X$ is a circle, the space $Y$ is an annulus, the map $f$ is not surjective, and the map $g$ is not injective.
Indeed, the circle and the annulus are not homeomorphic.
However, the maps $f$ and $g$ still encode a sense in which the two spaces are similar.
Indeed, note that in this example, we do have $g\circ f = \id_X$.
However, we have that $f\circ g\neq \id_Y$, since the entire (blue) line segment in the bottom left portion of $Y$ is mapped under $g$ to a single point of $X$, which is then mapped to a single point in $Y$.
Nevertheless, we do have that $f\circ g$ ``can be deformed'' to the identity map $\id_Y$, in a sense that we will make rigorous.

\begin{definition}
For topological spaces $X$ and $Y$ and maps $f_1, f_2 \colon X \to Y$, we say that $f_1$ and $f_2$ are \emph{homotopic}, denoted $f_1 \simeq f_2$, if there is a continuous function $F: X \times [0,1] \to Y$ such that $F(x, 0) = f_1 (x)$ and $F(x, 1) = f_2 (x)$ for all $x \in X$.
\end{definition}

\begin{definition}
Two topological spaces $X$ and $Y$ are homotopy equivalent, denoted $X \simeq Y$, if there exist continuous maps $f \colon X \to Y$ and $g \colon Y \to X$ such that $g \circ f \simeq \id_X$ and $f \circ g \simeq \id_Y$.
\end{definition}

In Figure~\ref{fig:homotopy}, we can see that although $f\circ g\neq \id_Y$, it is possible to continuously deform $\id_Y$ into $f\circ g$, as shown in Figure~\ref{fig:homotopy example}.
This means that $f \circ g \simeq \id_Y$.
Since we also have $g \circ f = \id_X$, this means $X\simeq Y$.
So, while the spaces shown in Figure~\ref{fig:homotopy} are not homeomorphic, they are homotopy equivalent.

\begin{figure}[htb]
\centering
\includegraphics[width=3in]{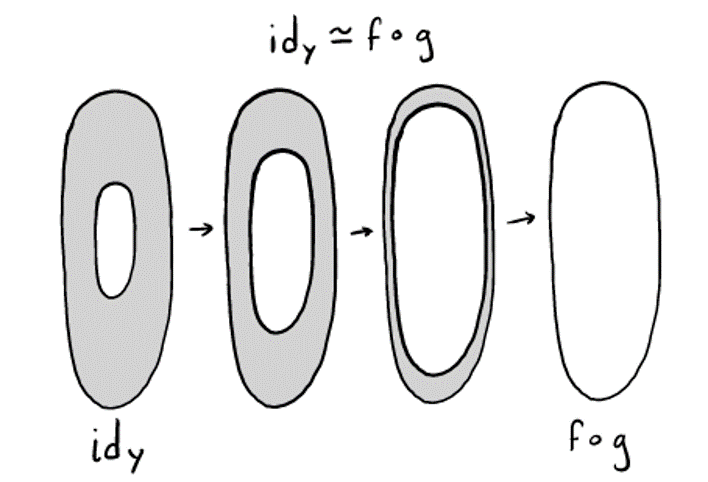}
\caption{An example of a homotopy from $\id_Y$ to $f\circ g$.}
\label{fig:homotopy example}
\end{figure}

\subsection{Time-Varying Spaces}
\label{sec:tv-equivalence}

In this section, we define time-varying spaces and maps between them.
In Subsections~\ref{subsubsection: Time Varying Homeomorphism} and~\ref{subsubsection: Time-Varying Homotopy Equivalence}, we will introduce homeomorphisms and homotopy equivalences between time-varying spaces.

For clarity, we will frequently be using the time interval $I=[0,1]$ in this section and throughout, though any of the following concepts can be applied to any interval $[a,b]$ such that $a \le b$.

\begin{definition}
\label{def: time-varying space}
A \emph{time-varying space $X$} is a topological space equipped with a continuous map to a time interval,  $q\colon X \to I$.
\end{definition}

\begin{definition}
\label{def: fiber}
At any time $t \in I=[0,1]$, we define $X_t \coloneqq q^{-1}(t)$ to be the \emph{fiber} of the time-varying space $X$ at time $t$.
\end{definition}

Figure~\ref{fig:time varying space} depicts the relations between the space, the time interval, the map to time, and a fiber of this map, as described in Definitions~\ref{def: time-varying space} and~\ref{def: fiber}.

We now define time-varying maps, which are designed to map from one time-varying space to another in a way that preserves the time coordinate.

\begin{figure}[htb]
\centering
\includegraphics[width=4in]{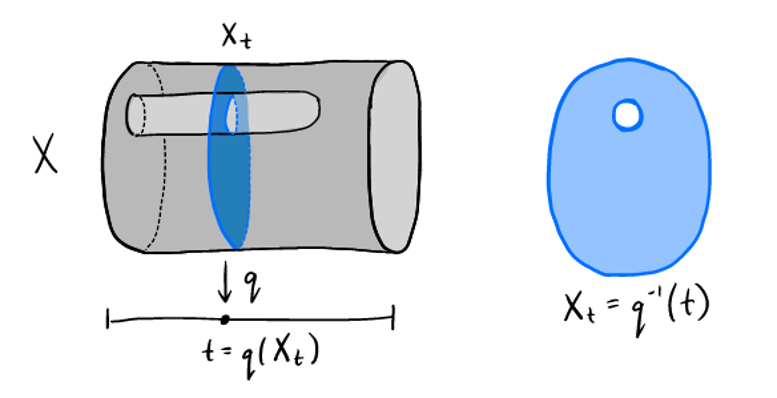}
\caption{A time-varying space $q\colon X\to I$ and one of its fibers.}
\label{fig:time varying space}
\end{figure}

\begin{definition}
\label{def: tv map}
Let $X$ and $Y$ be two time-varying spaces.
A function $f\colon X \to Y$ is a \emph{map between time-varying spaces} if $f$ is continuous and $f(X_t) \subseteq Y_t$ for all $t\in I$.
\end{definition}

We emphasize that time-varying maps are always continuous by definition.

\subsubsection{Time-Varying Homeomorphisms}
\label{subsubsection: Time Varying Homeomorphism}

Now that we have maps between time-varying spaces, we can define homeomorphisms between them as well.

\begin{definition}\label{def: time-varying homeomorphism}
A time-varying map $f \colon X \to Y$ is a \emph{time-varying homeomorphism} if f is a bijection and $f^{-1}$ is continuous.
If such a time-varying homeomorphism exists, then we say $X$ and $Y$ are \emph{time-varying homeomorphic}, denoted $X\cong_{tv} Y$.
\end{definition}

We remark that if $f \colon X \to Y$ is a time-varying homeomorphism, then we can restrict to get a standard homeomorphism $f_t \colon X_t \to Y_t$ at any time $t \in I$.

\begin{figure}[htb]
\centering
\includegraphics[width=3in]{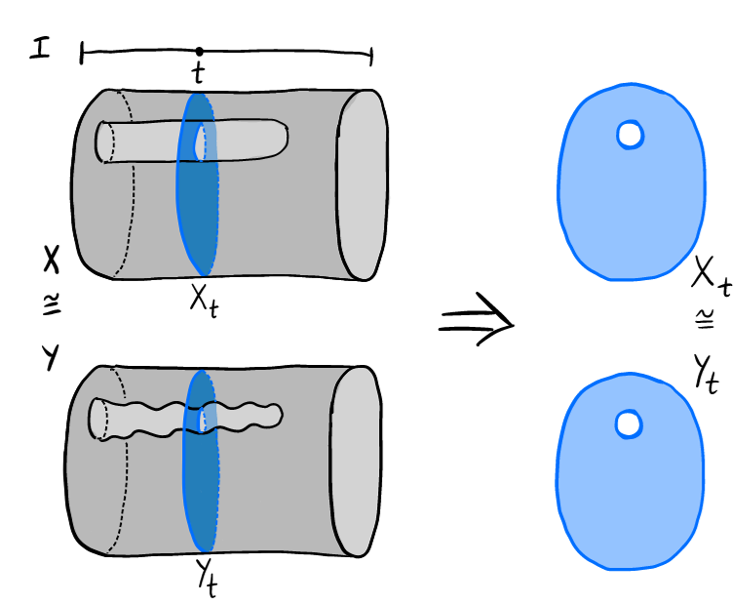}
\caption{An example time-varying homeomorphism.}
\label{fig:time varying homeomorphism}
\end{figure}

\subsubsection{Time-Varying Homotopy Equivalences}
\label{subsubsection: Time-Varying Homotopy Equivalence}

Let $J=[0,1]$ be the unit interval.
Recall that we have already defined $I=[0,1]$ to be the unit interval as well.
However, for time-varying maps, we use the symbol $J$ when we are thinking $[0,1]$ as parametrizing space, and we use the symbol $I$ when we are thinking of $[0,1]$ as parametrizing time.

\begin{definition}
If $Z$ and $W$ are time-varying spaces and $h,k \colon Z \to W$ are time-varying maps, then $h$ and $k$ are time-varying homotopic, denoted $h \simeq_{tv} k$, if there is a time-varying map $F \colon Z \times J \to W$ with $F(z,0)=h(z)$ and $F(z,1)=k(z)$ for all $z \in Z$.
\end{definition}

First of all, in order for $F$ to be a time-varying map, we need $Z\times J$ to have the structure of a time-varying space.
This is indeed the case: since $Z$ is a time-varying space, we have a map $q\colon Z\to I$, and then we can give $Z\times J$ the structure of time-varying space by defining the map $q' \colon Z\times J \to I$ via $q'(z,t)=q(z)$.
See Figure~\ref{fig:time varying interval}.

\begin{figure}[htb]
\centering
\includegraphics[width=4in]{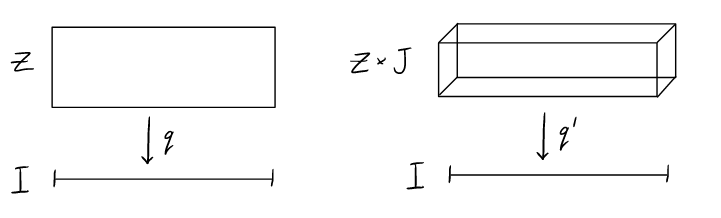}
\caption{Given a time-varying space $Z$ \emph{(left)}, we produce a time-varying space structure on $Z\times J$ \emph{(right)}.}
\label{fig:time varying interval}
\end{figure}

Here, $J$ encodes ``deformation space''; that is, for each $j \in J$, $F(z,j)$ produces another intermediate step in the deformation between $h(z)$ and $k(z)$.
See Figure~\ref{fig:homotopy-maps-and-def} for an example.

\begin{figure}[htb]
\centering
\includegraphics[width=4in]{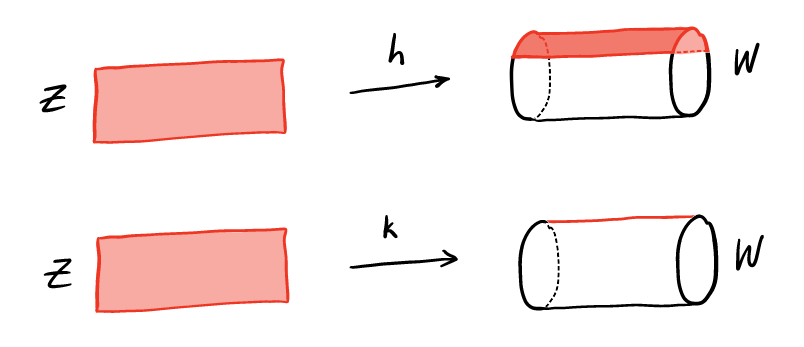}
\includegraphics[width=4in]{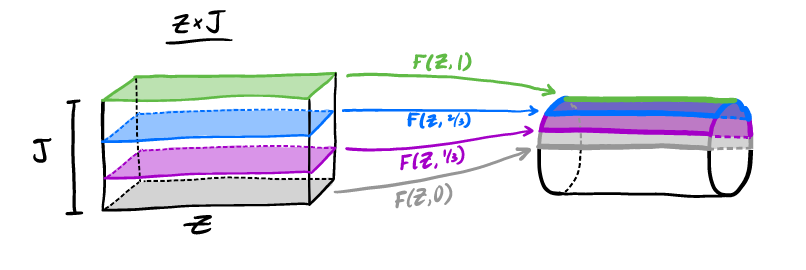}
\caption{(Top) The maps $h,k\colon Z\to W$.
(Bottom) A homotopy equivalence $F\colon Z\times J\to W$ that deforms from $h$ to $k$.}
\label{fig:homotopy-maps-and-def}
\end{figure}

\begin{definition}
\label{def:tv-he}
Two topological spaces $X$ and $Y$ are \emph{time-varying homotopy equivalent}, denoted $X \simeq_{tv} Y$, if there exist time-varying maps $f \colon X \to Y$ and $g \colon Y \to X$ such that $g \circ f \simeq_{tv} \id_X$ and $f \circ g \simeq_{tv} \id_Y$.
\end{definition}

\subsection{Pairs of Spaces}
\label{sec:pair-equivalence}

We now introduce the concept of \emph{pairs of spaces}, which we will use in Section~\ref{sec:application} to show that certain time-varying spaces are not time-varying homeomorphic.

\begin{definition}
A pair of spaces $X = (i \colon X_1 \hookrightarrow X_2)$ consists of two topological spaces $X_1 \subseteq X_2$ and an inclusion map $i \colon X_1 \hookrightarrow X_2$.
\end{definition}

For example, the topological spaces $X \hookrightarrow Y$ in Figure~\ref{fig:inclusion} form a pair of spaces.

\begin{figure}[htb]
\centering
\includegraphics[width=1.75in]{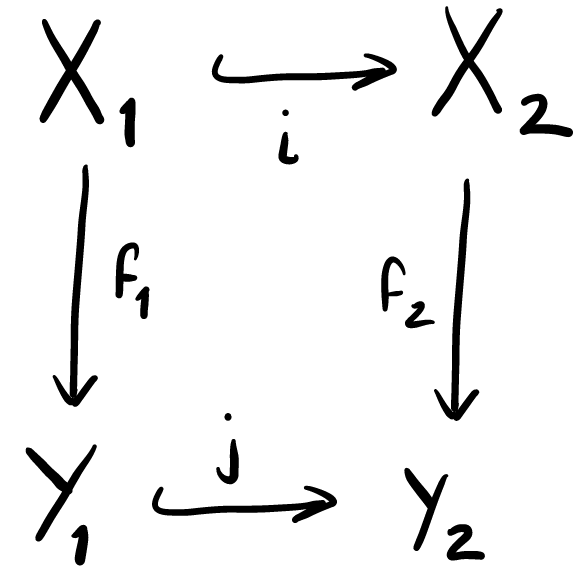}
\caption{Maps between pairs of spaces.}
\label{fig:maps between pairs}
\end{figure}

\begin{remark}\label{rmk: pairs of space from time varying space}
Let $X$ be a time-varying space and consider two times $t$ and $t'$.
As described in Definition~\ref{def: fiber}, this gives us two fibers of the time-varying space, $X_t$ and $X_{t'}$.
Suppose that, from time $t$ to time $t'$, we have an inclusion $X_t \subseteq X_{t'}$.
Then the inclusion map from time $t$ to $t'$ gives us a pair of spaces $X_t \hookrightarrow X_{t'}$.
\end{remark}

\begin{definition}
Suppose we have two pairs of spaces $X = (i \colon X_1 \hookrightarrow X_2)$ and $Y = (j \colon Y_1 \hookrightarrow Y_2)$.
A \emph{map $f\colon X \to Y$ between pairs of spaces} consists of the data $f=(f_1,f_2)$ where $f_1\colon X_1 \to Y_1$ is continuous, where $f_2\colon X_2 \to Y_2$, is continuous, and where $f_2 \circ i = j \circ f_1$.
We also consider a \emph{map $f\colon X \to Y$ between pairs of spaces up to homotopy}, which only needs to satisfy commutativity up to homotopy: $f_2 \circ i \simeq j \circ f_1$.
\end{definition}

See Figure~\ref{fig:maps between pairs} for the general diagram describing a map between pairs of spaces.
A specific example can be obtained by restricting Figure~\ref{fig:fibers} to the pairs of spaces $A_2 \hookrightarrow A_3$ and $B_2\hookrightarrow B_3$.

The following lemma will be used in our proof of Theorem~\ref{thm:not-homeomorphic} regarding two sensor networks that are time-varying homotopy equivalent but not time-varying homeomorphic.

\begin{lemma}\label{lem:map-between-pair-of-spaces-from-map-between-time-varying-spaces}
Let $f\colon X\to Y$ be a time-varying map, and fix times $t\le t'$.
Suppose that for all times $t\le s\le s'\le t'$ we have inclusions  $X_s \subseteq X_{s'}$ and $Y_s \subseteq Y_{s'}$, i.e.\ neither space grows smaller over the time interval $[t,t']$.
Then we have not only two pairs of spaces $X_t \hookrightarrow X_{t'}$ and $Y_t \hookrightarrow Y_{t'}$, but also a map between these two pairs of spaces up to homotopy.
\end{lemma}

\begin{proof}
For all $t\le s\le s'\le t'$, let $i_{s,s'} \colon X_s \hookrightarrow X_{s'}$ and $j_{s,s'}\colon Y_s \hookrightarrow Y_{s'}$ be the given inclusions from time $s$ to $s'$.
The map $f\colon X\to Y$ induces maps $f_s\colon X_s\to Y_s$ on the fibers.
We claim that $f_{t'} \circ i_{t,t'} \simeq j_{t,t'} \circ f_t$ (giving a map between pairs of spaces up to homotopy).
Indeed, note that the continuous map $F\colon X_t \times [t,t'] \to Y_{t'}$ defined by $F(x,s)=j_{s,t'}\circ f_s \circ i_{t,s}$ is a homotopy from $F(\cdot,t)=j_{t,t'} \circ f_t$ to  $F(\cdot,t')=f_{t'} \circ i_{t,t'}$.
\end{proof}

\section{Application to a sensor network evasion problem}
\label{sec:application}

Figure~\ref{fig:spaces} shows two time-varying spaces, A and B, with time progressing from left to right.
The shaded portion indicates the region in spacetime covered by sensors, and the unshaded portion indicates the uncovered region.

In the sensor network A, we begin with sensors covering the boundary $\partial D$, and with mobile sensors covering the bottom half of the domain $D$.
These sensors covering the bottom half retreat to the boundary, leaving only sensors on the boundary and a horizontal line of sensors.
Two sensors on this horizontal line move up into the top hole, and then these two sensors move a bit closer together, forming a new square hole.
The bottom two sensors of this hole move a little bit further apart, and the edge on the bottom of this square hole disappears.
The resulting squiggly line of sensors straightens out, and then sensors flood in from the top half of the boundary in order to cover the top half of the domain $D$.

\begin{figure}[htb]
\centering
\includegraphics[width=5in]{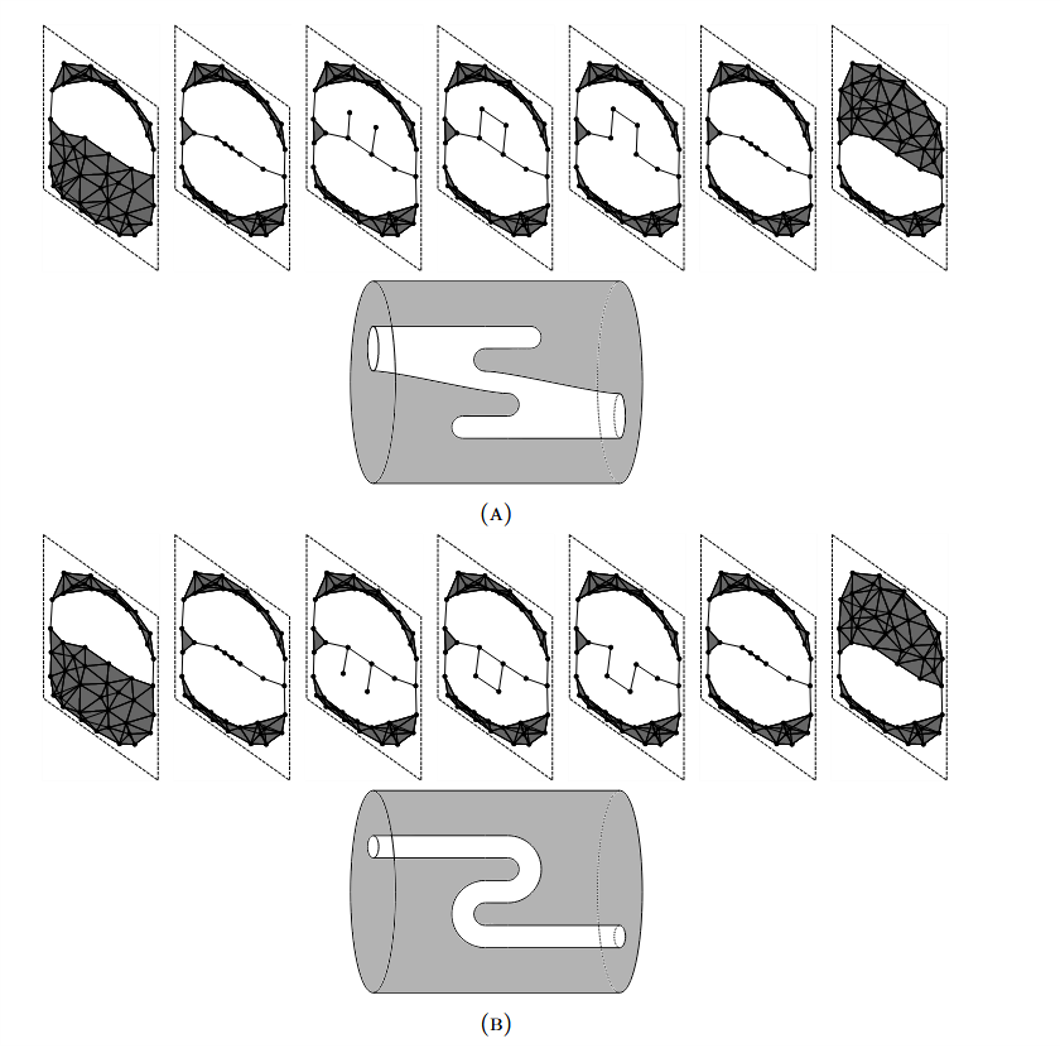}
\caption{Two time-varying spaces A and B, shown both as 7 snapshots in time, along with a cartoon beneath showing the covered region in spacetime.
Figure from~\cite{EvasionPaths}}
\label{fig:spaces}
\end{figure}

In the sensor network B, the motions of the sensors are identical, except that the small square hole opens up below the horizontal line of sensors.
Note that the sensor networks A and B measure \v{C}ech complexes that are combinatorially identical at all times $t$.
Nevertheless, the sensor network A has an evasion path in its uncovered region, whereas the sensor network B does not (since an intruder is not allowed to ever move backwards in time).
See~\cite{EvasionPathsVideo} for a video showing the motion of the sensors in these two examples.

Our goals are as follows: 
\begin{itemize}
\item In Subsections~\ref{subsection:homeomorphism} we will demonstrate that there is no time-varying homeomorphism between the two spaces A and B.
(It is worth noting that the two spaces are time-varying homotopy equivalent, but their complements are not.)
\item In Subsection~\ref{subsection:homotopy-equivalence}, we will demonstrate that the covered regions of the two spaces A and B are time-varying (fiberwise) homotopy equivalent.
This fact was stated in~\cite{EvasionPaths}, but here we elaborate on the proof.
\item In Subsection~\ref{subsection:uncovered-not-homotopy-equivalent}, we will show that the complements of these spaces --- the uncovered regions of $A$ and $B$ --- are \emph{not} time-varying homotopy equivalent.    One uncovered region has an evasion path (a continuous time-varying map from the interval in) whereas the other doesn't.
If the two uncovered regions were time-varying homotopy equivalent, one would have an evasion path if and only if the other did.
\end{itemize}

\begin{figure}[htb]
\centering
\includegraphics[width=5in]{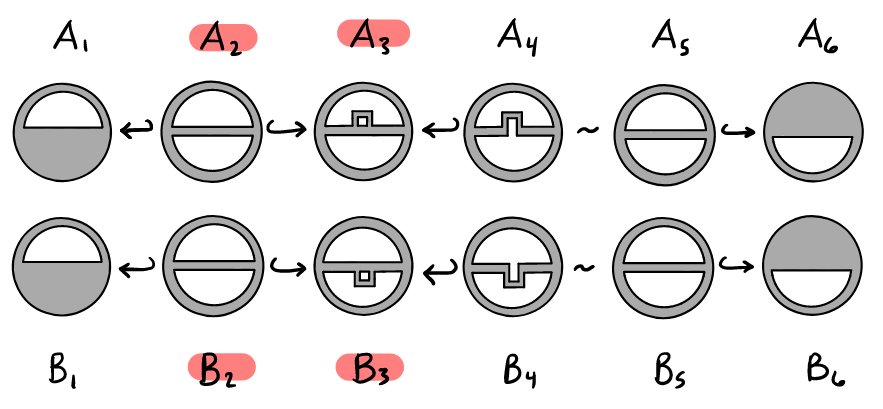}
\caption{Fibers of the time-varying spaces $A$ and $B$.}
\label{fig:fibers}
\end{figure}

There is an analogy that can be made with knot theory.
In knot theory, any two knots are circles, but they may be embedded in $\mathbb{R}^3$ in different ways, and have different complements (see for example the Gordon--Luecke Theorem~\cite{gordon1989knots}).
In a similar way, the two covered regions in A and B are the same time-varying space; they are time-varying homotopy equivalent.
However, they are mapped into spacetime in different ways, and they have different complements that are not time-varying homotopy equivalent (as one has an evasion path though the other does not).
So, we think of the sensor network examples A and B as two different ``spacetime knots'' formed by mapping the same space into spacetime in two different ways.

\subsection{Homeomorphism}\label{subsection:homeomorphism}

We will show that the covered regions of the sensor networks A and B are not time-varying homeomorphic.
First, we can examine the fibers of $A$ and $B$, as described in Definition~\ref{def: fiber}.
Figure~\ref{fig:fibers} depicts the fibers of $A$ and $B$ that we are specifically interested in; we label these specific fibers as $A_1,\ldots,A_7$ and $B_1,\ldots,B_7$.
When we show that there is no time-varying homeomorphism between these sensor network examples, we first pass to a commutative diagram between pairs of spaces up to homotopy.

\begin{theorem}
\label{thm:not-homeomorphic}
The sensor networks A and B in Figure~\ref{fig:spaces} are not time-varying homeomorphic.
\end{theorem}

\begin{figure}[htb]
\centering
\includegraphics [width=3in]{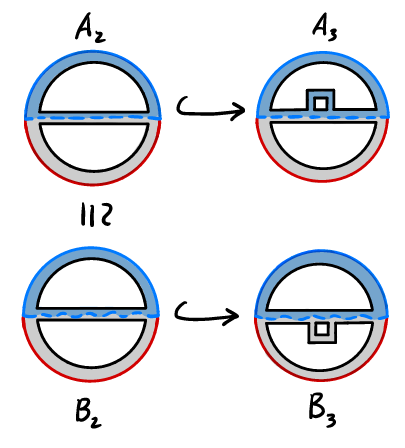}
\caption{Lack of homeomorphism between $A_3^{top}$ and $B_3^{top}$.}
\label{fig:halves}
\end{figure}

Let us first give an intuitive argument for why Theorem~\ref{thm:not-homeomorphic} is true.
Suppose for a contradiction that we have a time-varying homeomorphism $f\colon A\to B$ that restricts to the identity map on the boundary region $\partial D \times I$ covered by the fence sensors.
We have inclusions $i\colon A_2 \hookrightarrow A_3$ and $j\colon B_2 \hookrightarrow B_3$, and furthermore the fibers are only growing larger for all intermediate times in the intervals between these fibers.
The homeomorphism $f$ restricts to homeomorphisms $f_2 \colon A_2\to B_2$ and $f_3\colon A_3\to B_3$ on each fiber.
Let $A_2^{top}$ represent the top half of the fiber $A_2$; see Figure~\ref{fig:halves} where $A_2^{top}$ is shaded (in blue).
By Lemma~\ref{lem:map-between-pair-of-spaces-from-map-between-time-varying-spaces}, this diagram commutes up to homotopy, i.e.\ $f_3\circ i\simeq j\circ f_2$ as maps from $A_2$ to $B_3$.
But as we will show, this will contradict the fact that applying the homeomorphism $f_3$ to $i(A_2)$ maps onto loops in the fundamental group of $B_3$ that are not homotopic to loops in the image of $j$ applied to $f_2(A_2)$.
We now make this argument rigorous.

\begin{figure}[htb]
\centering
\includegraphics [width=3in]{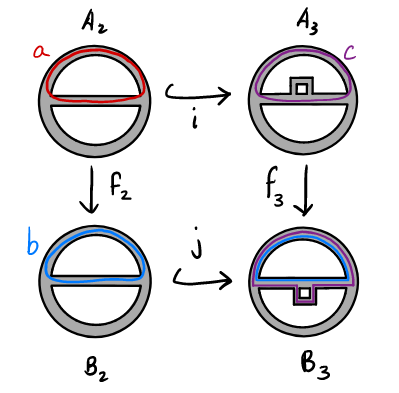}
\caption{Maps between selected fibers of $A$ and $B$.}
\label{fig:maps-between-fibers}
\end{figure}

\begin{proof}[Proof of Theorem~\ref{thm:not-homeomorphic}]
Suppose for a contradiction that we have a time-varying homeomorphism $f\colon A\to B$ that restricts to the identity map on the boundary region $\partial D \times I$ covered by the fence sensors.
We have inclusions $i\colon A_2 \hookrightarrow A_3$ and $j\colon B_2 \hookrightarrow B_3$, and furthermore the fibers are only growing larger for all intermediate times in the intervals between these fibers.
The homeomorphism $f$ restricts to homeomorphisms $f_2 \colon A_2\to B_2$ and $f_3\colon A_3\to B_3$ on each fiber.
Consider the loop $a$ in $A_2$, which maps under the inclusion $i$ to the loop $c$ in $A_3$.
The only homeomorphism $f_2\colon A_2\to B_2$ that restricts to the identity on the boundary $\partial D$ sends the loop $a$ in $A_2$ to a loop homotopic to the loop $b$ in $B_2$, which maps under the inclusion $j$ to a loop homotopic to the inner (blue) loop in $B_3$.
The only homeomorphisms $f_3\colon A_3\to B_3$ that restrict to the identity on the boundary $\partial D$ send the loop $c$ in $A_3$ to loops homotopic to the outer (purple) loop in $B_3$.
However, we can then see that in $B_3$, the loop $j(f_2(a))$ is not homotopy equivalent to $f_3(i(a))$, meaning that $j\circ f_2$ and $f_3 \circ i$ induce different maps on the fundamental group $\pi_1(A_2)\to\pi_1(B_3)$.
Hence $f_3\circ i\not\simeq j\circ f_2$.
This contradicts Lemma~\ref{lem:map-between-pair-of-spaces-from-map-between-time-varying-spaces}, which implies that this diagram commutes up to homotopy, i.e.\ $f_3\circ i\simeq j\circ f_2$.
\end{proof}

The reason why this proof does not preclude the possibility of having a time-varying homotopy equivalence $A\to B$ is that if $f_3$ is only a homotopy equivalence (instead of a homeomorphism) preserving $\partial D$, then $f_3$ could map the loop $c$ to the inner (blue) loop $j(b)$, as we will see in the next section.

\subsection{Homotopy Equivalence between Covered Regions}\label{subsection:homotopy-equivalence}

We now explain why the covered regions of the sensor networks A and B in Figure~\ref{fig:spaces} are in fact time-varying homotopy equivalent.

Recall from Definition~\ref{def:tv-he} that two topological spaces X and Y are \emph{time-varying homotopy equivalent}, denoted $X \simeq_{tv} Y$, if there exist time-varying maps $f \colon X \to Y$ and $g \colon Y \to X$ such that $g \circ f \simeq_{tv} \id_X$ and $f \circ g \simeq_{tv} \id_Y$.
We will show that such maps $f\colon A\to B$ and $g\colon B\to A$ exist for the sensor network spaces $A$ and $B$.

Let $A_t$ represent the union of the sensor balls in network $A$ at time $t$, and let $A'_t$ be the nerve of the sensor balls at time $t$, i.e.\ the \v{C}ech complex of the sensors at time $t$.
Similarly, let $B_t$ represent the union of the sensor balls in network $B$ at time $t$, and let $B'_t$ be the nerve of these sensor balls at time $t$.
We can think of the nerves, $A'_t$ and $B'_t$, as collections of vertices, edges, and triangles representing portions of the 2-dimensional covered region.

\begin{theorem}
\label{thm:covered-homotopy-equivalent}
The sensor networks A and B (see Figure~\ref{fig:spaces}) are time-varying homotopy equivalent.
\end{theorem}

\begin{proof}
For ease of exposition, let us first restrict attention to a fixed time $t$, and explain what the restricted maps $f_t\coloneqq f|_{A_t}\colon A_t \to B_t$ and $g_t\coloneqq g|_{B_t}\colon B_t \to A_t$ look like.

By the nerve lemma (Theorem~\ref{thm: nerve lemma}), we know that each fiber, $A_t$ or $B_t$, is homotopy equivalent to its nerve; $A_t \simeq A'_t$ and $B_t \simeq B'_t$.
In other words, we have maps $\alpha_t\colon A_t \to A'_t$ and $\alpha'_t\colon A'_t\to A_t$ with $\alpha'_t\circ \alpha_t\simeq \id_{A_t}$ and $\alpha_t\circ \alpha'_t\simeq \id_{A'_t}$.
Similarly, we have maps $\beta_t\colon B_t \to B'_t$ and $\beta'_t\colon B'_t\to B_t$ with $\beta'_t\circ \beta_t\simeq \id_{B_t}$ and $\beta_t\circ \beta'_t\simeq \id_{B'_t}$.
Furthermore, it is possible to choose these maps $\alpha_t, \alpha'_t, \beta_t, \beta'_t$ so that they vary continuously with $t$.

We also have homeomorphisms $A'_t \cong B'_t$ between the nerves of the covered regions for all $t$.
Let us denote this homeomorphism as $h_t\colon A'_t \to B'_t$ (and hence $h_t^{-1}\colon B'_t \to A'_t$.)
For many choices of $t$, this is clear; we can use the identity map to send $A'_t$ to $B'_t$.
We will examine fibers $A'_3$ and $B'_3$, and also $A'_4$ and $B'_4$, in more detail.
See Figure~\ref{fig:ft-chain-sketch}, which shows the homeomorphism $h_3\colon A'_3\to B'_3$ which ``reflects'' the small rectangular hole over a horizontal axis.
This is the main trick needed when constructing our time-varying homotopy equivalence.
The homeomorphism $h_4\colon A'_4\to B'_4$ is obtained by restricting $h_3$ to the subspace $A'_4\subseteq A'_3$, namely $h_4=h_3|_{A'_4}$.
It is furthermore possible to choose these homeomorphisms $h_t$ so that they vary continuously with $t$.

\begin{figure}[htb]
\centering
\includegraphics[width=5in]{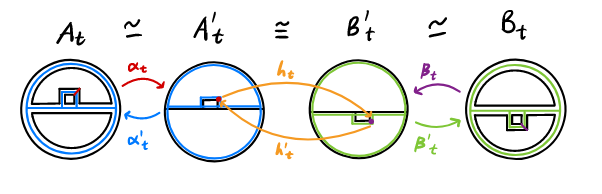}
\caption{Figure illustrating the proof of Theorem~\ref{thm:covered-homotopy-equivalent}.}
\label{fig:ft-chain-sketch}
\end{figure}

For each $t\in I$, we define the map $f_t\colon A_t \to B_t$ between the fibers of the sensor networks $A$ and $B$ as $f_t = \beta'_t \circ h_t \circ \alpha_t$.
Similarly, we define the map $g_t\colon B_t \to A_t$ as $g_t = \alpha'_t \circ h^{-1}_t \circ \beta_t$.
Note that $g_t \circ f_t \simeq \id_{A_t}$, since
\begin{align*}
    g_t \circ f_t &= (\alpha'_t \circ h^{-1}_t \circ \beta_t) \circ (\beta'_t \circ h_t \circ \alpha_t) \\
    &= \alpha'_t \circ h^{-1}_t \circ \beta_t \circ \beta'_t \circ h_t \circ \alpha_t \\
    &\simeq \alpha'_t \circ h^{-1}_t \circ h_t \circ \alpha_t &&\text{since }
    \beta_t\circ \beta'_t\simeq \id_{B'_t} \\
    &= \alpha'_t \circ \alpha_t \\
    &\simeq \id_{A_t}.
\end{align*}
Similarly, note that $f_t \circ g_t \simeq \id_{B_t}$, since
\begin{align*}
    f_t \circ g_t &= (\beta'_t \circ h_t \circ \alpha_t) \circ (\alpha'_t \circ h^{-1}_t \circ \beta_t) \\
    &= \beta'_t \circ h_t \circ \alpha_t \circ \alpha'_t \circ h^{-1}_t \circ \beta_t \\
    &\simeq \beta'_t \circ h_t \circ h^{-1}_t \circ \beta_t &&\text{since }
    \alpha_t\circ \alpha'_t\simeq \id_{A'_t} \\
    &= \beta'_t \circ \beta_t \\
    &\simeq \id_{B_t}.
\end{align*}
Next, we piece together the maps on the fibers in order to define $f\colon A\to B$ and $g\colon B\to A$.
Indeed if $a\in A$, then there is a unique $t$ with $a\in A_t\subseteq A$, in which case we let $f(a) = f_t(a) \in B_t \subseteq B$.
We define $g \colon B\to A$ similarly: if $b\in B_t \subseteq B$, then let $g(b)=g_t(b) \in A_t \subseteq A$.
The maps $f$ and $g$ are continuous because the maps $f_t$ and $g_t$ to vary continuously in $t$.
Similarly, the homotopy equivalences $g_t \circ f_t \simeq \id_{A_t}$ and $f_t \circ g_t \simeq \id_{B_t}$ piece together in a continuous way in order to give time-vary homotopy equivalences $g \circ f \simeq \id_{A}$ and $f \circ g \simeq \id_{B}$, i.e.\ giving that $A \simeq_{tv} B$.

In summary, we showed that $A' \cong_{tv} B'$ via continuously-varying homeomorphisms $A'_t \cong B'_t$ for all $t$.
Similarly, we showed $A \simeq_{tv} A'$ and $B' \simeq_{tv} B$ via continuously-varying homotopy equivalences $A_t \simeq A'_t$ and $B_t \simeq B'_t$ for all $t$.
Composing these together, we obtained
\[A \simeq_{tv} A' \cong_{tv} B' \simeq_{tv} B.\]
Thus $A \simeq_{tv} B$, as desired.
\end{proof}

The key reason \emph{why} we were able to show the spaces $A$ and $B$ are time-varying homotopy equivalent, even though they are not time-varying homeomorphic, is as follows.
After we performed a homotopy equivalence to contract the two-dimensional space $A_t$ down to the nerve complex $A'_t$ in Figure~\ref{fig:ft-chain-sketch} that is partially one-dimensional, and similarly for the homotopy equivalence $B_t \simeq B'_t$, we were able to see that the nerves $A'_t$ and $B'_t$ were homeomorphic through a map that preserved the fence sensors. 
However, if we hadn't yet shrunk things down to be one-dimensional (using a homotopy equivalence that is not a homeomorphism), then we'd be trying to find a homeomorphism between the 2-dimensional spaces $A_t$ and $B_t$ along the lines of Figure~\ref{fig:not-tv-homeomorphic}.

\begin{figure}[htb]
\centering
\includegraphics [width=4in]{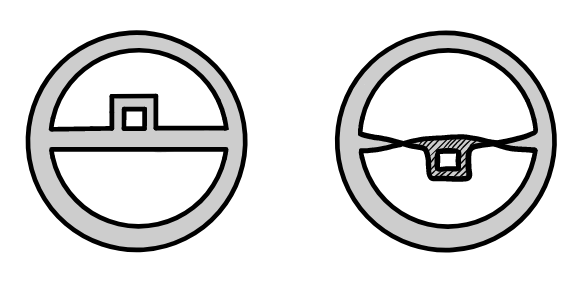}
\caption{
An attempted time-varying homeomorphism from $A$ to $B$ goes astray (shown here restricted to the fibers at time $3$).}
\label{fig:not-tv-homeomorphic}
\end{figure}

Figure~\ref{fig:not-tv-homeomorphic} is evocative of \emph{ribbon graphs} or \emph{fat graphs}, which (roughly speaking) are graphs equipped with a cylic ordering of the edges around each vertex.
In this language, the two graphs in Figure~\ref{fig:not-tv-homeomorphic} are equivalent as graphs, but not equivalent as ribbon graphs.
A ribbon graph structure determines at most one embedding in the 2-sphere up to isotopy, i.e.\ at most one embedding into the plane $\mathbb{R}^2$ up to isotopy if you fix which cycle is the ``fence'' (the outer cycle).
One can interpret Figure~\ref{fig:not-tv-homeomorphic} as a cartoon illustration for why one cannot construct a time-varying homeomorphism from sensor network $A$ to sensor network $B$: even though the sensor networks $A$ and $B$ are time-varying homotopy equivalent, and even though the fibers $A_3$ and $B_3$ are homeomorphic, no time-varying homeomorphism from $A$ to $B$ can preserve the ribbon graph structure (cyclic orderings of edges around each vertex) on each fiber.
One could also use this to construct a proof that no time-varying homeomorphism from $A$ to $B$ exists.
For further references on ribbon graphs, we point the reader to~\cite{Igusa,MoharThomassen}, and for their application to mobile sensor network problems we refer the reader to~\cite[Section~8]{EvasionPaths},~\cite[Section~3.3]{adams2021efficient}, and~\cite{distributed}.

\subsection{Uncovered Regions are not Homotopy Equivalent}\label{subsection:uncovered-not-homotopy-equivalent}

Though the \emph{covered} regions of $A$ and $B$ are time-varying homotopy equivalent, the \emph{uncovered} regions are not.
Indeed, the uncovered region for A has an evasion path, whereas the uncovered region for B does not.
And, as the following theorem shows, the complement of a time-varying sensor network that has an evasion path cannot be time-varying homotopy equivalent to the complement of one that does not.

First, we recall some notation.
If a sensor network has covered region $X\subseteq D\times I$ in spacetime, then the corresponding uncovered region is $X^c = (D \times I) \setminus X$.
An \emph{evasion path} is a time-varying map $I \to X^c\subseteq D\times I$.

\begin{theorem}
\label{thm:evasion-path}
Let $X,Y\subseteq D\times I$ be two time-varying covered regions.
If $X^c$ has an evasion path $I \to X^c$ and $Y^c$ does not, then the two uncovered regions $X^c$ and $Y^c$ are not time-varying homotopy equivalent.
\end{theorem}

\begin {proof}
There exists an evasion path for $X^c$, i.e.\ there exists a continuous time-varying map $\iota \colon I \to X^c$.
By Definition~\ref{def:tv-he}, we know that if the spaces $X^c$ and $Y^c$ are time-varying homotopy equivalent, then we have a continuous time-varying map $f\colon X^c\to Y^c$.
Note then that $f\circ\iota\colon I \to Y^c$ is then an evasion path in $Y^c$.
This shows that if the uncovered regions $X^c$ and $Y^c$ are time-varying homotopy equivalent, and if $X^c$ has an evasion path, then so does $Y^c$.
\end{proof}

\begin{corollary}
If one has two sensor networks $A$ and $B$, in which $A$ has an evasion path in the complement but $B$ does not, then the uncovered regions for $A$ and $B$ cannot be time-varying homotopy equivalent.
\end{corollary}

\section{Conclusion}
\label{sec:conclusion}

We describe notions of equivalence for time-varying topological spaces, and we use this machinery to compare two examples of spaces (A and B) arising from mobile sensor networks.
We show that the topological spaces arising from the covered regions of two mobile sensor networks are not time-varying homeomorphic, even though they are time-varying homotopy equivalent.

We prove that the sensor networks $A$ and $B$ are not time-varying homeomorphic by showing that there exist times $t_1$ and $t_2$ and inclusion maps between the fibers of $A$ and $B$ at those times that do nor commute with possible homeomorphisms.
To make this rigorous, we demonstrate the lack of homeomorphism by showing that there is no way to induce consistent maps on the fundamental groups of these fibers.

By the nerve lemma, the fibers of our spaces are homotopy equivalent to their nerves (\v{C}ech complexes).
We then show that $A'_t$ and $B'_t$, the nerve complexes of $A$ and $B$, are homeomorphic for all $t \in I$.
Using this equivalence of the nerve complexes, we prove that the covered regions \emph{are} time-varying homotopy equivalent, via the chain of equivalences $A \simeq_{tv} A' \cong_{tv} B' \simeq_{tv} B$.
The equivalence in the middle relies heavily on the fact that portions of the nerve complexes of the spaces are one-dimensional, even though the fibers they represent are two-dimensional.
Roughly speaking, the two-dimensional fibers carry extra ``ribbon graph structure'' that distinguish $A$ and $B$ up to fiberwise homeomorphism.
Ribbon graphs are a potential area for further exploration.
Lastly, we show that the uncovered regions are \emph{not} time-varying homotopy equivalent, because the uncovered region of one network has an evasion path while the other does not.

Future work might address the following open questions.

\begin{question}
Suppose $X$ and $Y$ are the covered regions for two time-varying sensor networks and that $X$ and $Y$ are time-varying homotopy equivalent; i.e., $X\simeq_{tv} Y$.
Suppose there is an evasion path in the complement of $X$, but no evasion path in the complement of $Y$.
Then, is it necessarily the case that $X$ and $Y$ are not time-varying homeomorphic?
\end{question}

\begin{question}
Can weak rotation data be incorporated with the time-varying \v{C}ech complex, instead of the alpha complex (which is more difficult to measure) to determine if-and-only-if an evasion path exists?
See the Open Question on Pages~29 and~109 of~\cite{MyThesis} and~\cite{EvasionPaths}, respectively, for precise statements of this question, which remains open a decade later.
\end{question}

\bibliographystyle{plain}
\bibliography{TimeVaryingSpacesAndMobileSensorNetworks.bib}

\begin{thebibliography}{10}

\bibitem{MyThesis}
Henry Adams.
\newblock {\em Evasion paths in mobile sensor networks}.
\newblock PhD thesis, Stanford University, 2013.

\bibitem{EvasionPathsVideo}
Henry Adams and Gunnar Carlsson.
\newblock The sensor network evasion problem: {D}ependence on the embedding.
\newblock YouTube video \url{https://youtu.be/K5gA7-EezXU?si=weABzTJQaqPWTtJL}.

\bibitem{EvasionPaths}
Henry Adams and Gunnar Carlsson.
\newblock Evasion paths in mobile sensor networks.
\newblock {\em International Journal of Robotics Research}, 34(1):90--104,
  2015.

\bibitem{adams2021efficient}
Henry Adams, Deepjyoti Ghosh, Clark Mask, William Ott, and Kyle Williams.
\newblock Efficient evader detection in mobile sensor networks.
\newblock {\em arXiv preprint arXiv:2101.09813}, 2021.

\bibitem{armstrong2013basic}
Mark~Anthony Armstrong.
\newblock {\em Basic topology}.
\newblock Springer Science \& Business Media, 2013.

\bibitem{Borsuk1948}
Karol Borsuk.
\newblock On the imbedding of systems of compacta in simplicial complexes.
\newblock {\em Fundamenta Mathematicae}, 35(1):217--234, 1948.

\bibitem{carlsson2020space}
Gunnar Carlsson and Benjamin Filippenko.
\newblock The space of sections of a smooth function.
\newblock {\em arXiv preprint arXiv:2006.12023}, 2020.

\bibitem{cavanna17when}
Nicholas~J. Cavanna, Kirk~P. Gardner, and Donald~R. Sheehy.
\newblock When and why the topological coverage criterion works.
\newblock In {\em Proceedings of the ACM-SIAM Symposium on Discrete
  Algorithms}, pages 2679--2690, 2017.

\bibitem{chintakunta2014distributed}
Harish Chintakunta and Hamid Krim.
\newblock Distributed localization of coverage holes using topological
  persistence.
\newblock {\em IEEE Transactions on Signal Processing}, 62(10):2531--2541,
  2014.

\bibitem{crabb2012fibrewise}
Michael~Charles Crabb and Ioan~Mackenzie James.
\newblock {\em Fibrewise homotopy theory}.
\newblock Springer Science \& Business Media, 2012.

\bibitem{Coordinate-free}
Vin de~Silva and Robert Ghrist.
\newblock Coordinate-free coverage in sensor networks with controlled
  boundaries via homology.
\newblock {\em The International Journal of Robotics Research},
  25(12):1205--1222, 2006.

\bibitem{de2007coverage}
Vin de~Silva and Robert Ghrist.
\newblock Coverage in sensor networks via persistent homology.
\newblock {\em Algebraic \& Geometric Topology}, 7(1):339--358, 2007.

\bibitem{EdelsbrunnerHarer}
Herbert Edelsbrunner and John~L Harer.
\newblock {\em Computational Topology: An Introduction}.
\newblock American Mathematical Society, Providence, 2010.

\bibitem{gamble2015coordinate}
Jennifer Gamble, Harish Chintakunta, and Hamid Krim.
\newblock Coordinate-free quantification of coverage in dynamic sensor
  networks.
\newblock {\em Signal Processing}, 114:1--18, 2015.

\bibitem{gordon1989knots}
C~McA Gordon and John Luecke.
\newblock Knots are determined by their complements.
\newblock {\em Journal of the American Mathematical Society}, 2(2):371--415,
  1989.

\bibitem{Hatcher}
Allen Hatcher.
\newblock {\em Algebraic Topology}.
\newblock Cambridge University Press, Cambridge, 2002.

\bibitem{Igusa}
Kiyoshi Igusa.
\newblock {\em Higher Franz-Reidemeister Torsion}, volume~31.
\newblock American Mathematical Society, 2002.

\bibitem{Kerber2013}
Michael Kerber and Herbert Edelsbrunner.
\newblock 3d kinetic alpha complexes and their implementation.
\newblock In {\em 2013 Proceedings of the Fifteenth Workshop on Algorithm
  Engineering and Experiments ({ALENEX})}, pages 70--77. Society for Industrial
  and Applied Mathematics, January 2013.

\bibitem{distributed}
Denis Khryashchev, Jie Chu, Mikael Vejdemo-Johansson, and Ping Ji.
\newblock A distributed approach to the evasion problem.
\newblock {\em Algorithms}, 13(139):1--13, 2020.

\bibitem{MoharThomassen}
Bojan Mohar and Carsten Thomassen.
\newblock {\em Graphs on Surfaces}, volume~2.
\newblock Johns Hopkins University Press, Baltimore, 2001.

\end{thebibliography}

\appendix

\end{document}